\documentclass[12pt,reqno]{amsart}

\usepackage{amssymb}
\usepackage{amsfonts}
\usepackage{mathrsfs}
\usepackage{amsfonts}
\usepackage{amscd}
\usepackage{amsmath}

 \numberwithin{equation}{section}
\newtheorem{lemma}{Lemma}[section]
\newtheorem{theorem}[lemma]{Theorem}
\newtheorem{remark}[lemma]{Remark}

\newtheorem{coro}[lemma]{Corollary}
\newtheorem{definition}[lemma]{Definition}
\newtheorem{example}[lemma]{Example}

\title{State space decomposition\\ for nonautonomous dynamical systems}

\author{Xiaopeng ~Chen}
\address[X. Chen]{%
 School of Mathematics and Statistics\\ Huazhong University of Science and
Technology\\ Wuhan 430074, China} \email[X.
Chen]{chenxiao002214336@yahoo.cn}
\author{Jinqiao ~Duan}
\address[J. Duan]{%
Department of Applied Mathematics
\\ Illinois Institute of Technology\\ Chicago, IL 60616, USA
}
\email[J.~Duan]{duan@iit.edu}

\date{\today }


\begin{document}
\begin{abstract}
Decomposition of   state spaces into dynamically different
components  is helpful for the understanding of dynamical
behaviors of complex systems.
 A Conley type decomposition theorem is proved for
nonautonomous dynamical systems defined on a non-compact but
separable state space. Namely, the state space can be decomposed
into   a chain recurrent part and a gradient-like part.

This  result applies  to both nonautonomous ordinary differential
equations on Euclidean space (which is only locally compact),  and
  nonautonomous partial differential equations on infinite
dimensional function space (which is not even locally compact).
This decomposition result is demonstrated by discussing a few
concrete examples, such as the Lorenz system and the Navier-Stokes
system, under time-dependent forcing.
\end{abstract}

\maketitle

\section{Introduction} \label{intro}

The decomposition of state spaces for  dynamical systems or flows
is desirable for better understanding of dynamical behaviors. The
Conley decomposition theorem \cite{Con} says that any flow on a
\emph{compact} state space decomposes the space into a chain
recurrent part and a gradient-like part. The theorem describes the
dynamical behavior of each point in the systems. It is considered
as a fundamental theorem of dynamical systems \cite{Nor}.

There are two essential concepts in the consideration of Conley
decomposition of state spaces. One is the chain recurrence  set.
Conley  \cite{Con} showed that the  chain recurrent set
$CR(\varphi)$ for a dynamical system $\varphi$ on a compact state
space can be represented in terms of complement sets of (local)
attractors. This result is widely studied and further extended by
others in different contexts \cite{Cho, Chu, ColoniusJohnson, Hur,
Hur1, Hur2, Pat} or for random dynamical systems \cite{Liu1, Liu2,
Liu3}.

The other essential concept is the so-called complete Lyapunov
function, which quantifies gradient-like behavior. A complete
Lyapunov function for a dynamical system $\varphi$ is a continuous,
real-valued function $L$ defined on the state space which is
strictly decreasing on orbits outside the chain redurrent set and
such that: (a) The range $L(CR(\varphi))$ is nowhere dense; (b) If
$c$  belongs to the range  $L(CR(\varphi))$, then $L^{-1}(c)$ is a
component of the chain recurrent set. We call the complement of the
chain recurrent set the gradient-like part of the flow. For more
details see \cite{Con}. Furthermore, the complete Lyapunov function
can be extended to non-compact state spaces for deterministic
dynamical systems \cite{Hur1, Hur3, Pat, Pat1}  or random dynamical
systems \cite{Liu2, Liu3}. Especially, in
  \cite{Liu2}, the base space needs to be separable to
construct  the countable local attractors. While in \cite{Liu3},
it is the weak complete Lyapunov function for the random
semiflows.

In the present paper, we consider Conley type decomposition for
nonautonomous dynamical systems (NDS), defined on  not necessarily
compact  state spaces. Recall that a nonautonomous dynamical
system is defined in terms of a cocycle mapping on a state space
that is driven by an \emph{autonomous} dynamical system acting on
a base space. More details about nonautonomous dynamical systems
are reviewed in the next section. The standard examples of
nonautonomous dynamical systems are those generated by
nonautonomous ordinary or partial differential equations, which
  arising from modeling in biological, physical and environmental systems.
We will prove the following main result.

\begin{theorem}\label{1.2} (Conley decomposition for NDS). \\
A nonautonomous dynamical system with a separable (but not
necessarily compact) state space  decomposes the space into a
chain recurrent part and a gradient-like part.
\end{theorem}
Here the ``gradient-like part" for the NDS is indicated by the
complete Lyapunov function. Such function is constructed using
  special attractor-repeller pairs. It is known that the
attractor-repeller pairs are  basic notions for the definition of
Morse decompositions \cite{Ochs,Ras,Ras1}. The Conley type state
space decomposition for NDS can be applied to both nonautonomous
ordinary differential equations on Euclidean space (which is only
locally compact),  and
  nonautonomous partial differential equations on infinite
dimensional function space (which is not even locally compact). In
Section \ref{appl}, we illustrate this result by discussing a few
concrete examples, such as the Lorenz system and the Navier-Stokes
system, under time-dependent forcing.

To prove the above decomposition theorem, we first  define and
investigate the chain recurrent set for a nonautonomous dynamical
system. In this context, a local attractor   is a pullback
attractor when it is a nonempty compact subset in the state space.
The relationship of different attractors for nonautonomous
dynamical systems is considered in \cite{Che1, Joh}. In the case
of nonautonomous ordinary or partial differential equations, we
can consider chain recurrent set for the corresponding
skew-product dynamical system \cite{Che, Pea}. It is known that
the global attractor for a skew-product dynamical system
corresponds to the pullback attractor on the state space
\cite{Che1, Che3, Yej}. We prove a similar relation for   local
attractors (Lemma \ref{3.10}) and apply to the chain recurrent
set. Then, we consider the complete Lyapunov function for NDS.
This concept of Lyapunov functions   is weaker than that for
autonomous dynamical systems. Note that the base space here does
not need to be separable.

This paper is organized as follows. After reviewing basic facts
for nonautonomous dynamical systems (NDS) in Section \ref{nds}, we
investigate
 chain recurrent sets and complete Lyapunov functions for NDS
 in Section \ref{chain} and Section \ref{Lya}, respectively.
The nonautonomous decomposition Theorem \ref{1.2} is thus proved.
In Section 5 we present a few examples, both ordinary and partial
differential equations, to demonstrate the decomposition result.

\section{Prelimineries} \label{nds}

We will use the symbol $\mathbb{T}$ for either $\mathbb{R}$ or
$\mathbb{Z}$, and  denote by $\mathbb{T}^+$ all non-negative
elements of $\mathbb{T}$.  Let $dist_X$ denote the Hausdorff
semi-metric between two nonempty sets of a metric space $(X,d_X)$,
that is
\begin{equation}   dist_X(A,B)=\sup\limits_{a\in A}\inf\limits_{b\in
B}d_X(a,b),
\end{equation}
for $A\subset X, B\subset X$. In addition, if  A or B are empty,
we set $dist_X(A, B)= 0$. We recall some basic definitions for
 nonautonomous dynamical systems \cite{Che1,Ras,Sch} on state space   $X$
 with base space (also a metric space) $P$.

\begin{definition}(Nonautonomous Dynamical System (NDS)). An
autonomous dynamical system $(P,\mathbb{T},\theta)$ on $P$ consists
of a continuous mapping $\theta_t:\mathbb{T}\times P\rightarrow  P$
for which the $\theta_t=\theta(t,\cdot):P\rightarrow P$, $t\in
\mathbb{T}$, form a group of homeomorhpisms on $P$ under composition
over $\mathbb{T}$, that is, satisfy
\begin{equation*}
\theta_0=id_P, \quad \quad \theta_{t+s}=\theta_t\cdot\theta_s
\end{equation*}
for all $t,s\in \mathbb{T}$. In addition, a continuous mapping
$\varphi:\mathbb{T}^+\times P\times X\rightarrow X$ is called a
cocycle with respect to an autonomous dyanmical system
$(P,\mathbb{T},\theta)$ if it satisfies
\begin{equation*}
\varphi(0,p,x)=x,\quad
\varphi(t+s,p,x)=\varphi(t,\theta_sp,\varphi(t,p,x))
\end{equation*} for all $t,s\in \mathbb{T}^+$ and $(p,x)\in P \times
X$.
\end{definition}
The triple $<X,\varphi,(P,\mathbb{T},\theta)>$ is called a
nonautonomous dynamical system \cite{Che1,Sch}. Let
$(\mathbb{U},d_\mathbb{U})$ be the cartesian product of $(P,d_P)$
and $(X,d_X)$. Then the mapping $\pi: \mathbb{T}^+\times
\mathbb{U}\rightarrow \mathbb{U}$ defined by
\begin{equation*}
  \pi(t,(p,x)):=(\theta_tp,\varphi(t,p,x))
\end{equation*}
forms a semi-group on $\mathbb{U}$ over $\mathbb{T}^+$; see
\cite{Sel}.

 A subset $M$ of  $\mathbb{U}$ is called a nonautonomous set.
 Let $M(p):=\{x\in X: (p,x)\in M\} \mbox{ for } \; p\in P$.
 A nonautonomous set  $M$ is called closed, compact or open if
  $M(p), p\in P$, are closed, compact or open, respectively.
A nonautonomous set $M$ is called forward invariant if $\varphi(t,p,
M(p))\subset M(\theta_tp)$ for all $t\in \mathbb{T}^+$, $p\in P$ and
backward invariant if $\varphi(t,p, M(p))\supset M(\theta_tp)$ for
all $t\in \mathbb{T}^+$, $p\in P$. A nonautonomous set $M$ is called
invariant if $\varphi(t,p, M(p))=M(\theta_tp)$ for all $t\in
\mathbb{T}^+$ and $p\in P$.

Let $\mathscr{D}$ be a family of sets $(D(p))_{p\in P}$ in $X$ such
that if $(D(p))_{p\in P}\in \mathcal {D}\subset \mathscr{D} $,
$(D'(p))_{p\in P}\in  \mathscr{D} $ and $D'(p)\subset D(p)$, then
$(D'(p))_{p\in P}\in \mathcal {D} $.

\begin{definition}(Pullback attractor for NDS).\\
Let $<X, \varphi, (P, T, \theta)>$ be a nonautonomous dynamical
system. An element $A(p)\in \mathcal {D}$ such that $A(p)$ is
compact is called pullback attractor (with respect to $\mathcal
{D}$) if the invariance property
\begin{equation*}   \varphi(t,
p, A(p))=A(\theta_tp), \quad \quad t\in \mathbb{T}^+, p\in P
\end{equation*}
and the pullback convergence property
\begin{equation*}
\lim \limits_{t\rightarrow \infty}dist_X(\varphi(t,\theta_{-t}p,
D(\theta_{-t}p )), A(p))=0
\end{equation*}
for all $(D(p))_{p\in P}\in \mathcal {D}$ are fulfilled.
\end{definition}
In the following we recall the definitions of local attractor and
chain recurrent set, which are important to  the state space
decomposition.
\begin{definition}(Local attractor for NDS).\\
  An open set $U(p)$ is called a  pre-attractor if it satisfies
\begin{equation}\label{2.1}
\overline{\bigcup \limits_{t\geq
\tau(p)}\varphi(t,\theta_{-t}p)U(\theta_{-t}p)}\subset U(p)
\quad\quad \mbox{for some } \tau(p)>0,
\end{equation}where $\tau$ is a function.   We define the
local attractor $A(p)$ inside $U(p)$ as follows:
\begin{equation}\label{2.2}
  A(p)=\bigcap\limits_{n\in \mathbb{N}}\overline{\bigcup\limits_{s\geq
  n\tau(p)}\varphi(s,\theta_{-s}p)U(\theta_{-s}p)}.
\end{equation}
The basin of attraction $B(A,U)(p)$, determined by $A(p)$ and
$U(p)$, is defined as follows:
\begin{equation}
  B(A,U)(p)=\{x:\varphi(t,p)x\in U(\theta_{t}p) \mbox{ for some } t\geq
  0\}.
\end{equation}
\end{definition}
It can be proved that if the time $\mathbb{T}$ is two-sided and
the state space $X$ is compact, then for any $D(p)\subset
B(A,U)(p)$, we have
\begin{equation*}
\lim \limits_{t\rightarrow \infty}dist_X(\varphi(t,\theta_{-t}p,
D(\theta_{-t}p )), A(p))=0.
\end{equation*}
For a given local attractor $A(p)$, we define the repeller,
corresponding to $A(p)$, as $R(p):=X-B(A,U)(p)$. We call the pair
$(A, R)$ an attractor-repeller pair. Observe that
attractor-repeller pair depends on the pre-attractor $U(p)$. We
allow $A(p)=\emptyset$ or $R(p)=\emptyset$.  We use $F(P\times X)$
to denote the set of all maps from $P\times X$ to $\mathbb{R}^+$
and is continuous at fixed $p\in P$.

\begin{definition}(Chain recurrent set for NDS). \newline
(\em a) For a given  $ \varepsilon\in F(P\times X), T(p)>0$, the
sequence $\{x_1(p), \cdots,x_n(p), x_{n+1}(p);t_1,$ $t_2,$ $\cdots,
t_n\}$ is called an $(\varepsilon, T)(p)$-chain for $\varphi$ from
$x(p)$ to $y(p)$ if for $1\leq i\leq n$
\begin{equation*}
x_1(p)=x(p), \quad x_{n+1}(p)=y(p), \quad t_i\geq T(p).
\end{equation*}
and
\begin{equation*}
  d_X(\varphi(t_i,\theta_{-t_i}p)x_i(\theta_{-t_i}p),x_{i+1}(p))<\varepsilon(p,\varphi(t_i,\theta_{-t_i}p)x_i(\theta_{-t_i}p)),
\end{equation*}
where $x_i(p)$ is the map  from $P$ to $X$. \\
(\em b) A  map $x$ from $P$ to $X$ is called chain recurrent if
there exists an $(\varepsilon, T)(p)$-chain beginning and ending at
$x(p)$ for any $ \varepsilon\in F(P\times X),
 T(p)>0$.

(\em c) We denote $CR_{\varphi}(p)$ the chain recurrent set for
$\varphi$, i.e.,
\begin{equation*}
 CR_{\varphi}(p)=\{x(p)\mid x(p) \mbox{ is  chain recurrent variable  } \}.
\end{equation*}
\end{definition}
\begin{example}
Recall that a mapping $\gamma^*:P\rightarrow \mathcal {G}$ is called
a generalized fixed point of the cocycle $\Phi$
 if
\begin{equation*}
 \Phi(t,p,\gamma^*(p))=\gamma^*(\theta_tp) \mbox{ for
 }  t\in \mathbb{R}^+.
\end{equation*}
 Such a generalized fixed point for the NDS $\Phi$ is  chain recurrent with respect to $P$(for more details see \cite{Dua}).
\end{example}
\begin{definition}\label{2.3}(Stationary solution for NDS).\\
A solution x(p) is called stationary for NDS $\varphi$ if
$\varphi(t,p,x(p))=x(\theta_tp)$.
\end{definition}
From the Definition \ref{2.3} we know that the stationary solution
for NDS is in the chain recurrent set. In fact, we know from the
definition that  a generalized fixed point for NDS is also a
stationary solution for NDS. This is the same as  the random case
\cite{Sch1}. If $x(p)$ and $y(p)$ are chain recurrent with respect
to $P$. We say $x(p)\sim y(p)$ if and only if for each
$\varepsilon\in F(P\times X)$, $T(p)>0$ there is an
$(\varepsilon,T)$(p)-chain from $x(p)$ and $y(p)$ and one from
$y(p)$ to $x(p)$ for $p\in P$. The equivalence classes are called
the chain transitive components of $\varphi$.

\begin{definition} \label{complete}
  (Complete Lyapunov function for  NDS)\\
   A complete Lyapunov
  function for a NDS  $\varphi$  is a  function $L : P\times X\mapsto
  \mathbb{R}^+$,  with $L(p,\cdot)$ being  continuous for  $p\in P$, that
  satisfies the following conditions:\\
({\em a}) If $x\in CR_\varphi(p)$, then
\begin{equation*}
 L(\theta_tp,\varphi(t,p)x)=L(p,x), \quad \quad \forall t>0;
\end{equation*}
({\em b}) If $x\in X-CR_\varphi(p)$, then
\begin{equation*}
 L(\theta_tp,\varphi(t,p)x)\leq L(p,x), \quad \quad \forall t>0;
\end{equation*}
({\em c}) The range of $L(p,\cdot)$ on $CR_\varphi(p)$ is a compact
nowhere dense
subset of $[0,1]$; \\
({\em d}) If $x(p)$ and $y(p)$ belong to the same chain transitive
component of $\varphi$, then
$$L(p,x(p))=L(p,y(p)).$$ And if $x(p)$ and $y(p)$ belong to the different chain transitive
components, then
$$L(p,x(p))\neq L(p,y(p)).$$
\end{definition}

\section{Chain recurrent sets for nonautonomous dynamical
systems} \label{chain}

First we present some Lemmas related to the NDS.
\begin{lemma}
Assume that $U(p)$ is a given pre-attractor and $\bigcup
\limits_{t\geq \tau(p)}\varphi(t,\theta_{-t}p)U(\theta_{-t}p)$ is
pre-compact. Then $A(p)$ is a pullback attractor with respect to
$\mathcal {D}=\{D(p)\mid  {D(p)} $  is a closed nonempty subset of
$ U(p)$ $\mbox{ for every }  $ $p\in P\}$.
\end{lemma}
\begin{proof}
Denote $U(\tau(p))=\overline{\bigcup \limits_{t\geq
\tau(p)}\varphi(t,\theta_{-t}p)U(\theta_{-t}p)}$. It follows that
\begin{equation}\label{3.6}
    A(p)=\bigcap\limits_{n\in \mathbb{N}}U(n\tau(p)).
\end{equation}
Therefore  $A(p)$ is a nonempty compact set. The invariance of
$A(p)$ can be shown as in the   proof of   invariance for the
$\Omega$-limit set of a random set   \cite{Cra1}.

Moreover, for any $D(p)\in \mathcal{D}$,
\begin{equation*}
\lim\limits_{t\rightarrow
\infty}dist_X(\varphi(t,\theta_{-t}p)D(\theta_{-t}p), A(p))=0,
\end{equation*}
which shows that $A(p)$ is the pullback attractor with respect to
$\mathcal {D}$.
\end{proof}

\begin{remark} In the
definition of \eqref{2.2}, the local attractor needs not be
compact and it can be allowed as an empty set. It is known that
 the attractor is also noncompact in \cite{Aul}. But when the state
space is compact,
 the local attractor is the pullback attractor.
 For more details about pullback attractor
   see \cite{Car,Che1,Sch}.
\end{remark}

In the random case, the pre-attractor can be selected as a forward
invariant open set and repeller is a forward invariant closed set
\cite{Liu1,Liu2}. The following   Lemma \ref{3.13}, which can be
proved as in \cite{Liu1, Liu2}, shows that the basin of attraction
is a backward invariant open set and the repeller is forward
invariant closed set.



\begin{lemma}\label{3.13}
The basin of attraction $B(A,U)$ is a backward invariant open set
and the repeller $R$ is a forward invariant closed set.
\end{lemma}

\begin{lemma}\label{3.4}
If  $x(p)\in B(A,U)(p)$ and $x(p)$ is the chain recurrent variable
for $p\in
  P$, then   $x(p)\in A(p)$ for  $p\in P$, where $B(A,U)(p)$
 is the basin of attraction determined by $U(p)$ and $A(p)$.
\end{lemma}

\begin{proof} The idea is from \cite{Cho,Hur2}.
 For $\tau(p)$ satisfies \eqref{2.1}, we have
\begin{eqnarray*}
\bigcup \limits_{t\geq \tau(p)}\varphi(t,\theta_{-t}p)
U(\theta_{-t}p)\subset U(p).
\end{eqnarray*}
We need to show that there exists an $\varepsilon(p,x)>0$ such that
$\varepsilon(p,x) \leq 1$ and
\begin{eqnarray}
B(\varphi(t,\theta_{-t}p)x(\theta_{-t}p),
\varepsilon(p,\varphi(t,\theta_{-t}p)x(\theta_{-t}p)))\subset U(p)
\label{3.1}
\end{eqnarray}
 for all $x(p)\in U(p)$ and $t \geq \tau(p)$.

Let us construct such a  function $\varepsilon$. Define
$\delta(p,x)$ by
\begin{eqnarray*}
\delta(p,x)=\frac 1 2 \{d_X(x,{\bigcup \limits_{t\geq
\tau(p)}\varphi(t,\theta_{-t}p) U(\theta_{-t}p})+d_X(x, X-U(p))\}.
\end{eqnarray*}
Then $\delta(p,x)>0$ since $x\notin X-U(p)$ if $x\in {\bigcup
\limits_{t\geq \tau(p)}\varphi(t,\theta_{-t}p)
U(\theta_{-t}p)}\subset U(p)$. Let $x(p)\in U(p)$ and $t\geq
\tau(p)$. For any $y(p)\in B(\varphi(t,\theta_{-t}p)x(\theta_{-t}p),
\delta(p,\varphi(t,\theta_{-t}p)x(\theta_{-t}p)))$,
\begin{eqnarray*}
d_X(\varphi(t,\theta_{-t}p)x(\theta_{-t}p),y(p))<\delta(p,\varphi(t,\theta_{-t}p)x(\theta_{-t}p))=\frac
1 2 d_X(\varphi(t,\theta_{-t}p)x(\theta_{-t}p), X-U(p)).
\end{eqnarray*}
Thus we have
\begin{eqnarray*}
&&2d_X(\varphi(t,\theta_{-t}p)x(\theta_{-t}p),y(p))<d_X(\varphi(t,\theta_{-t}p)x(\theta_{-t}p),
X-U(p))\\& &\leq
d_X(\varphi(t,\theta_{-t}p)x(\theta_{-t}p),y(p))+d_X(y(p), X-U(p)).
\end{eqnarray*}
Since $d_X(y(p),X-U(p))>
d_X(\varphi(t,\theta_{-t}p)x(\theta_{-t}p),y(p))\geq 0$, we have
$y(p) \in U(p)$. Hence
\begin{eqnarray*}
B(\varphi(t,\theta_{-t}p)x(\theta_{-t}p),
\varepsilon(p,\varphi(t,\theta_{-t}p)x(\theta_{-t}p)))\subset U(p).
\end{eqnarray*}
and $\varepsilon=\min\{\delta,1\}$ is the desired function.

 Let  $m$, $n$ be positive integers.
Select $\varepsilon >0$ satisfies \eqref{3.1} for $x(p)\in U(p)$,
$p\in P$. If $x(p)\in CR_\varphi(p)$, there is an $(\frac
{\varepsilon} n , m{\tau})(p)$-chain $\{x_1(p),\cdots,x_k(p),$ $
x_{k+1}(p);$$ t_1, \cdots,t_k\}$ from $x(p)$ back to $x(p)$.  Since
\begin{eqnarray*}
&&d(\varphi(t_1,\theta_{-{t_1}}p )\circ
x_1(\theta_{-t_1}p),x_2(p))\\
&<&\frac 1 n \varepsilon(p,\varphi(t_1,\theta_{-{{{t_1}}}}p)
x_1(\theta_{-{t_1}}p))\\
&\leq & \varepsilon(p,\varphi(t_1,\theta_{{{-{t_1}}}}p
)x_1(\theta_{-{t_1}}p)).
\end{eqnarray*}
we have
\begin{eqnarray*}
x_2(p)\in  B(\varphi
({t_1},\theta_{-{t_1}}p)x(\theta_{-{t_1}}p),\varepsilon(p,\varphi
(t_1,\theta_{-{t_1}}p)x(\theta_{-{t_1}}p))) \subset U(p)
\end{eqnarray*}
by  \eqref{3.1}. Thus $x_k(p)\in U(p)$ by induction. Since
\begin{eqnarray*}
d(\varphi(t_k,\theta_{-{t_k}}p )x_k(\theta_{-{t_k}}p),x_{k+1}(p))
<\frac 1 n \varepsilon(p,\varphi(t_k,\theta_{-{t_k}}p)
x_k(\theta_{-{t_k}}p))\leq \frac 1 n,
\end{eqnarray*}
we obtain
\begin{eqnarray*}
d(x(p),\bigcup \limits_{t\geq m{\tau}(p)}\varphi(t,\theta_{-t}p)
U(\theta_{-t}p)) \leq d(x(p),\varphi(t_k,\theta_{-{t_k}}p
)x_k(\theta_{-{t_k}}p)) < \frac 1 n.
\end{eqnarray*}
Thus
\begin{eqnarray*}
d(x(p),\bigcup \limits_{t\geq m{\tau}(p)}\varphi(t,\theta_{-t}p)
U(\theta_{-t}p) )=0.
\end{eqnarray*}
It follows that  $x(p)\in Cl{\bigcup \limits_{t\geq
m{\tau}(p)}\varphi(t,\theta_{-t}p) U(\theta_{-t}p)} $. This implies
\begin{eqnarray*}
x(p)\in \bigcap \limits_{m\in \mathbb{N}} Cl{\bigcup \limits_{t\geq
m{\tau}(p)}\varphi(t,\theta_{-t}p) U(\theta_{-t}(p))}=A(p).
\end{eqnarray*}

Now
 $x(p)\in B(A,U)(p)$, $p\in
  P$, which implies there exists $s\geq0$ such that
\begin{eqnarray}\label{3.15}
  \varphi(s,p)x(p)\in{U( \theta_{s}p)}
\end{eqnarray}
for fixed $p\in P$. By a same method that given in \cite{Chu,Hur2},
we can conclude that $CR_\varphi$ is forward invariant. Hence $
\varphi(s,p)x(p)\in CR_\varphi(\theta_sp)$. Combining the above
prove process we have
\begin{eqnarray}\label{3.16}
\varphi(s,p)x(p)\in{A( \theta_{s}p)}.
\end{eqnarray}
 Let
$Y=\{t\geq 0\mid \pi(t,(p,x(p)))\in U\}$ and  $Z=\{t\geq 0\mid
\pi(t,(p,x(p)))\in A\}$. From \eqref{3.15} and \eqref{3.16} we know
that $Y\neq \emptyset$ and $Z\neq \emptyset$. By the continuity of
$\pi$, $Y$ is open in $[0,+\infty)$, while $Z$ is closed. So
$Y=Z=[0,+\infty)$, which shows  $x(p)\in A(p)$.
\end{proof}

We now present the following result on chain recurrent set.
\begin{theorem}\label{3.8}
  (Chain recurrent set for NDS)\\
   Let
  $U(p)$ be an arbitrary  pre-attractor,  $A(p)$ be the local
  attractor determined by $U(p)$,  and $B(A,U)(p)$ be the basin of
  attraction determined by $U(P)$ and $A(p)$. Then the following
  decomposition holds:
  $$X-CR_\varphi(p)=
\bigcup  [B(A,U)(p)-A(p)],$$
  where the union is taken over all local attractors  $A(p)$ determined by
  pre-attractors.
\end{theorem}

\begin{proof}
   Suppose $x$ is a
map from $P$ to $X$,  the function $\varepsilon(p,x)>0$ and
$\tau(p)>0$. Define
 \begin{eqnarray*}
 U_1(p):&=&\bigcup \limits_{t\geq
 \tau(p)}B(\varphi(t,\theta_{-t}p)x(\theta_{-t}p),\varepsilon(p,\varphi(t,\theta_{-t}p)x(\theta_{-t}p))),\\
 U_2(p):&=&\bigcup \limits_{t\geq \tau(p)}\bigcup \limits_{y(p)\in
 U_1(p)}B(\varphi(t,\theta_{-t}p)y(\theta_{-t}p),\varepsilon(p,\varphi(t,\theta_{-t}p)y(\theta_{-t}p))),\\
             &\cdots&\\
 U_n(p):&=&\bigcup \limits_{t\geq \tau(p)}\bigcup \limits_{y(p)\in
 U_{n-1}(p)}B(\varphi(t,\theta_{-t}p)y(\theta_{-t}p),\varepsilon(p,\varphi(t,\theta_{-t}p)y(\theta_{-t}p))),\\
             &\cdots.&
 \end{eqnarray*}
Since   $ U_n(p)$ $(n\geq 1)$ are all
 open sets. So the set
\begin{eqnarray}
U_x(p):=\bigcup \limits_{n\in \mathbb{N}}U_n(p)\label{3.2}
\end{eqnarray} is an  open set.
From the construction of $U_x(p)$ we   see that $U_x(p)$ is the
set of all possible end points of $(\varepsilon, \tau)(p)$-chains
that begin at $x(p)$. In the following we prove that $U_x(p)$ is a
pre-attractor and it determines a local attractor $A_x(p)$. Since
there exists a map $\delta: P\times X\rightarrow (0,+\infty)$ with
$\delta\leq \frac{\varepsilon} 2$ such that
$\varepsilon(p,y(p))>\frac 1 2 \varepsilon(p,x(p))$ when
$d(x(p),y(p))<\delta(p, x(p))$.  For
$B(y(p),\delta(p,y(p)))\bigcap
\varphi(t,\theta_{-t}p)U_x(\theta_{-t}p)\neq \emptyset$, $t\geq
\tau(p)$. There exists $z(p)\in U_x(p)$ such that
\begin{equation*}
  d(\varphi(t,\theta_{-t}p)z(\theta_{-t}p),y(p))<\delta(p,y(p)),
  \quad   t\geq \tau(p).
\end{equation*}
Since
$d(\varphi(t,\theta_{-t}p)z(\theta_{-t}p),y(p))<\delta(p,y(p))$, we
have $\varepsilon(p,\varphi(t,\theta_{-t}p)z(\theta_{-t}p))>\frac 1
2 \varepsilon(p,y(p))$. Thus
\begin{eqnarray*}
  d(\varphi(t,\theta_{-t}p)z(\theta_{-t}p),y(p))<\delta(p,y(p))\leq \frac 1 2
  \varepsilon(p,y(p))<
  \varepsilon(p,\varphi(t,\theta_{-t}p)z(\theta_{-t}p)).
\end{eqnarray*}
Hence there exists an $(\varepsilon,\tau)(p)$-chain from $x(p)$ to
$y(p)$. This means that $$\overline{\bigcup\limits_{t\geq
\tau(p)}\varphi(t,\theta_{-t}p)U_x(\theta_{-t}p)}\subset U_x(p).$$

 If $x(p)\in X - CR_\varphi
 (p)$,
 then for arbitrary $\varepsilon(p,x)>0$ and  $\tau(p)>0$  there
 exists no
$(\varepsilon,\tau)(p)$-chain  begins and ends at $x(p)$. Take $U_x$
defined by (\ref{3.2}), then by the construction of $U_x$, it is
easy to see that $x(p) \in B(A_x,U_x)(p)$
 and $ x(p)\notin U_x(p) $. Hence
$$x(p) \in  B(A_x,U_x)(p)-A_x(p)$$ for $p\in P$.
So
\begin{equation}\label{3.3}
 X -CR_\varphi(p) \subset \bigcup [B(A,U)(p)-A(p)].
\end{equation}

If  $x(p)$ is a chain recurrent variable and $x(p)\in B(A,U)(p)$,
then by Lemma \ref{3.4}, we have $x(p)\in A(p)$. Hence
\begin{equation*}
  x(p)\in  X - CR_\varphi(p) \quad \mbox{whenever} \quad  x(p)\in B(A,U)(p)-A(p).
\end{equation*}
So
\begin{equation}\label{3.5}
 \bigcup [B(A,U)(p)-A(p)] \subset X - CR_\varphi(p).
\end{equation}

 Therefore, by \eqref{3.3} and \eqref{3.5}, we obtain that
 $$ X - CR_\varphi(p) = \bigcup [B(A,U)(p)-A(p)]. $$
\end{proof}

\begin{coro}\label{3.9}
  Assume that $U(p)$ is a pre-attractor, $A(p)$ is the
  local attractor determined by $U(p)$,  and $R(p)$ is the repeller
  corresponding to $A(p)$ with respect to $U(p)$. Then
  $$CR_\varphi(p)=\bigcap [A(p)\bigcup R(p)],$$
  where the intersection is taken over all local attractrs.
\end{coro}

\begin{lemma}\label{3.10}
Let $U$ and $A$ be   nonautonomous sets. If  $U$ is the
pre-attractor for the skew-product system,
  then
  $$
  A(p)=\bigcap \limits_{t\geq T}\overline {\bigcup \limits_{s\geq t}\varphi(s,\theta_{-s}p, {U(\theta_{-s}p)})}
  $$
is the local attractor with the property  $\bigcup \limits_{p\in
P} \{p \}\times
  A(p) \subset A$, where we denote $(p,\emptyset)=\emptyset$.
\end{lemma}

\begin{proof} Since $U$ is the pre-attractor for the skew-product system, there exists $T\geq 0$
such that $\overline{\bigcup \limits_{s\geq T} \pi(s,U)}\subset U$.
Suppose $y\in \overline{ \bigcup \limits_{s\geq T}
\varphi(s,\theta_{-s} p, U(\theta_{-s}p))}$, there exists $s_n\geq
T$ and $x_n\in U(\theta_{-s_n}p)$ such that
{\setlength\arraycolsep{2pt}
\begin{eqnarray}\label{3.14}
(p,y)&=&\lim_{n\rightarrow \infty}(p,\varphi(s_n,\theta_{-s_n}p,
x_n))\nonumber\\
&=&\lim_{n\rightarrow \infty}\pi(s_n,(\theta_{-s_n}p,x_n))\nonumber \\
&\in & U=\bigcup\limits_{p\in P}\{p\}\times U(p).
\end{eqnarray}}
We thus conclude that $y\in U(p)$.  Thus $U(p)$ is the pre-attractor
and $A(p)$ is the corresponding local attractor.

Assume that $A=\bigcup \limits_{p\in P} \{p \}\times
  A'(p)$  with $A'(p)=\{x \mid (p,x)\in A \; \mbox{for fixed} \; p\}$.
  If $A(p)\neq \emptyset$,
  let the sequence $s_n\rightarrow +\infty$ in \eqref{3.14}.
   As in the   proof of \eqref{3.14}, we see that $A(p)\subset
  A'(p)$. If $A(p)=\emptyset$, we also have $A(p)\subset
  A'(p)$. Hence $\bigcup \limits_{p\in P} \{p \}\times
  A(p) \subset A$.
\end{proof}

\begin{coro}\label{3.11} With the convention
  $(p,\emptyset)=\emptyset$,   we have the following relation:
  $\bigcup \limits_{p\in P}\{p\}\times CR_{\varphi}(p) \subset CR(\pi)$.
\end{coro}

\begin{proof}
From Theorem \ref{3.8}, Lemma \ref{3.10} and the fact that
$B(A,U)=\bigcup \limits_{p\in P}\{p\}\times B(A,U)(p)$, the
required relation   follows.
\end{proof}

\section{Complete Lyapunov functions for nonautonomous dynamical
systems} \label{Lya}

If $X$ is a separable metric space, then we   take
$\{x_j\}_{j=1}^\infty$ as a countable dense subset  of $X$.  For
$p\in P$, we define $x_j(p)=x_j$,
$\varepsilon_j(p,x)=\varepsilon_j$ and $\tau_j(p)=\tau_j$, where
$\{\varepsilon_j\}_{j=1}^\infty\in \mathbb{Q}^+$ and
$\{\tau_j\}_{j=1}^\infty\in \mathbb{Q}^+$. By the construction of
\eqref{3.2}, we   obtain   countable local attractors which we
denote as $\{A_n(p)\}_{n=1}^\infty$.

\begin{lemma}\label{4.3}
  Assume that $(A(p), R(p))$ is a given attractor-repeller pair and $A(p)\in \{A_n(p)\}_{n=1}^\infty$. Then
  there exists a function $l$ for $(A(p),R(p))$ such
  that $l(p,x) $ is continuous with respect to $x\in X$  and   possesses
  the following properties:\\
  (i) $l(p,x)=0$ when $x\in A(p)$, and $l(p,x)=1$ when $x\in
  R(p)$, $p\in P$;\\
  (ii) For $\forall x\in X\backslash (A(p)\bigcup R(p))$ and for $\forall
  t>0$: $1>l(p,x)>l(\theta_tp,\varphi(t,p)x)>0$.
\end{lemma}
\begin{proof}
From \eqref{3.2}, we know that $U(\theta_sp)=U(p)$,
$A(\theta_sp)=A(p)$ and $R(\theta_sp)=R(p)$ for $s\in \mathbb{T}$.
Let
  $$
 \lambda(p,x):=\frac {dist_X(x,A(p))}{dist_X(x,A(p))+dist_X(x,R(p))}.  $$
 We set $dist_X(x,R(p))=1$ if $R(p)=\emptyset$, and $dist_X(x,A(p))=1$ if
 $A(p)=\emptyset$.

Then the function $\lambda$ is continuous with respect to $x\in X$
and
  \[ \lambda(p,x)=\begin{cases} 0,&x\in A(p); \\
1, &x\in R(p);\end{cases}
\]
Let $g(p,x)=\sup\limits_{t\geq
0}\lambda(\theta_tp,\varphi(t,p)x)$. In the following we prove
that $g(p,x)$ is continuous with respect to $x\in X$. Since $1\geq
g(p,x)\geq \lambda(p,x)=1$ for $x\in R(p)$, $g(p,x)$ is continuous
at $x\in R(p)$.  For $x\in X\backslash R(p)$,  there exists
$t_0\geq 0$ such that $\varphi(t_0,p,x)\in U(\theta_{t_0}p)$. We
also  have $\varphi(t_0,\theta_{-t-t_0}p,x)\in U(p)$ for $t\in
\mathbb{T}$. So
 \begin{eqnarray}\label{4.6}
\lim \limits_{t\rightarrow \infty}dist_X(\varphi(t,p,x),
A(\theta_tp))\leq \lim \limits_{t\rightarrow
\infty}dist_X(\overline{\varphi(t,\theta_{-t}p)U(\theta_{-t}p)},
A(p)) =0.
 \end{eqnarray}

This implies that
 $g(p,x)=\sup\limits_{t_0 \geq t\geq
0}\lambda(\theta_tp,\varphi(t,p)x)$ for some $t_0>0$. The
continuity of $g(p,x)$ for  fixed $p\in P$ follows from  the
 continuity of $g(p,x)=\sup\limits_{t_0 \geq t\geq
0}\lambda(\theta_tp,\varphi(t,p)x)$ for  fixed $p\in P$.  By
definition we have $g(\theta_tp,\varphi(t,p)x)\leq g(p,x)$.

The function $l$ is defined by $l(p,x)=\int
_0^{+\infty}e^{-t}g(\theta_tp,\varphi(t,p)x)dt$. Since $A(p)$ and
$R(p)$ are forward invariant,    we conclude that  $l(p,x)=0$ for
$x\in A(p)$, and $l(p,x)=1$ for $x\in R(p)$. If $x\in X\backslash
(A(p)\bigcup R(p))$, we define
$l(\theta_{t_1}p,\varphi(t_1,p)x)=l(p,x)$ for some $t_1>0$. Hence
$g(\theta_tp,\varphi(t,p)x)=g(\theta_{t+t_1}p,\varphi(t+t_1,p)x)$
for all $t\geq 0$, from which we   see  that
 $$g(x,p)=g(\theta_{nt_1}p,\varphi(nt_1,p)x)$$
 for all $n\in \mathbb{N}$.
 Let $n \rightarrow \infty$, we  arrive at a contradiction to
 \eqref{4.6}.
 Hence we have
$$l(\theta_tp,\varphi(t,\theta_{-t}p)x)<l(p,x). $$
Obviously, $l(p,x)$ is  continuous with respect to $x\in X$.
\end{proof}

\begin{remark}  In this proof, we have used a similar
construction as in  \cite{Liu3,Pat1,Ras}, which was originated
from \cite{Con}.
\end{remark}

Assume that $l_n(p,x)$ is the Lyapunov function determined by the
attractor-repeller pair $(A_n(p),$ $R_n(p))$. Define
\begin{equation}\label{4.2}
  L(p,x)=\sum\limits_{n=1}^\infty
\frac {2 l_n(p,x)}{3^n}.
\end{equation}
We now show that $L(p,x)$ defined by \eqref{4.2} is a complete
Lyapunov function for $\varphi$.

\begin{theorem} \label{Lyafunc}
  (Complete Lyapunov function for NDS)\\
   The function defined by
  \eqref{4.2} is a complete Lyapunov function for the NDS  $\varphi$.
\end{theorem}

\begin{proof}
We show that all conditions in Definition \ref{complete} are
satisfied.

(a) If $x(p)\in CR_\varphi(p)$, then we have
\begin{equation*}
  l_n(p,x(p))=l_n(\theta_tp, \varphi(t,p)x(p)), \quad \forall t\geq
  1,
\end{equation*}
which takes value $0$ or $1$ for each $n\in \mathbb{N}$ by
Corollary \ref{3.9} and Lemma \ref{4.3}.\\
(b)  By Lemma \ref{4.3}, we have
\begin{equation*}
  l_n(\theta_tp,\varphi(t,p)x)<l_n(p,x), \quad \quad \forall
  t>0,  n\in \mathbb{N}
\end{equation*}
for $x\in X\backslash (A_n(p)\bigcup R_n(p))$.  If $x\in A_n(p)\bigcup R_n(p)$, then
$\varphi(t,p)x\in A_n(\theta_tp)\bigcup R_n(\theta_tp)$.
So for $x\in X-CR_\varphi(p)$,
\begin{equation*}
  L(\theta_tp,\varphi(t,p)x)\leq L(p,x), \quad \quad \forall
  t>0,  n\in \mathbb{N}.
\end{equation*}
(c) This is due to the fact that $L(p, CR_\varphi(p))$ is a subset
of the Cantor middle-third
set.\\
(d) If $x(p)\in A_n(p)$, with $x(p)$ and $y(p)$ belonging to the
same chain transitive component, then $y(p)\in U_n(p)\subset
B(A_n,U_n)(p) $ for $p\in P$. Therefore $y(p)\in A_n(p)$ for $p\in
P$ by Lemma \ref{3.4}. Similarly,  if $x(p)\in R_n(p)$ then
$y(p)\in R_n(p)$. It is clear that $L(p,\cdot)$ is constant on
each chain transitive component and $L(p,\cdot)$ takes different
values on different transitive components.

This completes the proof.
\end{proof}

\begin{remark}
  The complete Lyapunov function obtained in Lemma \ref{Lyafunc} is
  weaker than that in the autonomous case.
  It is nonincreasing  along orbits of the skew-product flow
  $\pi(t,(p,x))$.
\end{remark}
\medskip

\textbf{Proof of Theorem \ref{1.2}:}

 Making use of the above
 Theorems \ref{3.8} and \ref{Lyafunc} for chain recurrent set and complete Lyapunov
 function, respectively,
we   obtain the   decomposition in Theorem \ref{1.2}.

\section{Applications} \label{appl}

In this section we consider a few   examples to illustrate the
applications of the  decomposition result in Theorem \ref{1.2}.

\begin{example}
  Consider the differential equation \cite{Klo}
\begin{equation}
    \dot{x}=2tx.
\end{equation}
The solution is $x(t,t_0, x_0)=x_0e^{t^2-t_0^2}$, where $t\geq
t_0$, and the cocycle mapping is
\begin{equation*}
  \varphi(t,t_0, x_0)=x_0e^{(t+t_0)^2-t_0^2},\quad \quad t\geq 0.
\end{equation*}
Here the driving space is $P=\mathbb{R}$ with element $p=t_0$, the
shift map is $\theta_tt_0=t+t_0$ and the state space is
$X=\mathbb{R}$. The pre-attractor is selected as $U=(-1,1)$. The
corresponding local attractor is $A(p)=\{0\}$. From the definition
of $\varphi$ we see that the basin of attraction is
$B(A,U)(p)=\mathbb{R}$. So the chain recurrent set is $\{0\}$.
Moreover, the complete Lyapunov function    is nonincreasing outside
$X-\{0\}$.
\end{example}

We revise the example in \cite{Cra} to fit our purpose here.
\begin{example}
  Consider the base space
      $P=S^1$ and the state space $X=S^1$. Define a shift map
  $\theta_t p=p+t$.
  Introduce a  homeomorphism $\psi(p):S^1\rightarrow S^1$ by
  $\psi(p)x=x+p$. We consider a  NDS defined by
  $$\varphi(t,p)=\psi(\theta_tp)\circ \varphi_0(t)\circ \psi^{-1}(p),$$
  where $\varphi_0$ is the semiflow on $S^1$ determined by the
  equation $$\dot{x}=-\cos x. $$
  Then the NDS $\varphi$ has no nontrivial local attractor.
  Hence by the  decomposition result in Theorem
\ref{1.2}, the chain recurrent
  set is $X$. This means that for any $x(p)\in X$, there exists an
  $(\varepsilon,T)(p)$-chain, beginning and ending at $x(p)$ for $p\in
  P$.
\end{example}

The following Lorenz system under time-dependent forcing was
considered in \cite{Che}.
\begin{example}
Let $\Omega$ be a compact metric space, $\mathbb{R}=(-\infty,
+\infty)$, $(\Omega,\mathbb{R},\sigma)$ be a dynamical system on
$\Omega$, and $H$ be a   Hilbert space. We denote $L(H)$ and $
L^2(H) $ as the spaces of all linear, bilinear, respectively,
endomorphisms on $H$. Let $C(\Omega,H)$ be the space of all
continuous functions $f:\Omega \rightarrow H$, endowed with the
topology of uniform convergence. Consider the non-autonomous
Lorenz system
\begin{equation}\label{4.7}
  u'=A(\omega t)u+B(\omega t)(u,u)+f(\omega t),\quad \omega \in
  \Omega,
\end{equation}
where $\omega t:=\sigma(t,\omega)$, $A\in C(\Omega,L(H))$, $B\in
C(\Omega,H)$,  and $f\in C(\Omega, H)$. Moreover, we assume
\eqref{4.7} satisfies the
following conditions: \\
(\em a) there exists $\alpha>0$ such that
$$Re\langle A(\omega) u,u \rangle \leq -\alpha|u|^2$$
for all $\omega \in \Omega$ and $u\in H$, where $|\cdot|$ is a norm
in $H$, generated by the scalar product $\langle \cdot,\cdot
\rangle$;\\
{(\em b)} $$Re\langle B(\omega)(u,v),w \rangle=-Re \langle
B(\omega)(u,w),v \rangle$$ for every $u,v,w\in H$ and $\omega\in
\Omega$. It can be shown that there exists a global solution
$\varphi(t,x,\omega)$  of \eqref{4.7} on $\mathbb{R}^+$, which
defines a nonautonomous dynamical system $\varphi$.

Recall that the dynamical system $(\Omega,\mathbb{R},\theta)$ is
called asymptotically compact if for any positively invariant
bounded set $A\subset X$, there is a compact $K_A\subset X$ such
that $$\lim \limits_{t\rightarrow
+\infty}dist_{\Omega}(\theta(t,A),K_A)=0.$$ The NDS $\varphi$ is
called asymptotic compact if the associated skew-product flow, on
  $P \times X$,   is asymptotic compact.

If $(\|f\|C_B)/ \alpha^2<1$, where
$C_B:=\sup\{|B(\omega)(u,v)|:\omega\in \Omega, u,v \in H, |u|\leq 1,
|v|\leq 1\}$ and the system \eqref{4.7} is asymptotic compact, then
there exists a map $x:\Omega\rightarrow H$ such that $x(\omega
t)=\varphi(t,x(\omega),\omega)$ for all $\omega\in \Omega$ and $t\in
\mathbb{R}^+$. Here we have $P=\Omega$ and $X=H$.  So the map $x$ is
chain recurrent with respect to $\Omega$. Therefore using the
decomposition result in Theorem \ref{1.2}, we can decompose the
space $H$ into two nontrivial parts. One is the chain recurrent part
and the other is the gradient-like part.

\end{example}

\begin{remark}
   Note that a stationary orbit (or stationary solution) is chain recurrent.
   If we know that there exists a  stationary solution
   in a nonautonomous differential
   equation, we can   conclude that   the
  chain recurrent set   is not empty.  This is one way to
  demonstrate  that the chain recurrent set is not empty.
\end{remark}
The following example is the nonautonomous Navier-Stokes equation,
which was also considered  in \cite{Che1, Pea}.

\begin{example} Consider the 2-dimensional Navier-Stokes equation
\begin{eqnarray}
 &&u_t+\sum\limits_{i=1}^2u_i\partial_iu=\nu\triangle u+\nabla p+F(t) \label{4.8}\\
  &&div u=0, \quad  u|_{\partial D}=0, \nonumber
\end{eqnarray}
where $u=(u_1,u_2)$ is the velocity field, $F(x,y,t)=(F_1,F_2)$ is
the external forcing, and $D$ is an open bounded fluid domain with
smooth boundary $\partial D\in C^2$. We denote by $H$ and $V$ the
closures of the linear space $\{u\mid u\in C_0^\infty (D)^2\}$ in
$L^2(D)^2$ and $H_0^1(D)^2$, respectively. We also denote by $Pr$
the corresponding orthogonal projection $Pr:L^2(D)^2\rightarrow H$.
We further set
\begin{equation*}
A:=-\nu Pr\triangle, \quad B(u,u):=Pr(\sum
\limits_{i=1}^2u_i\partial v)
\end{equation*}
  Applying the orthogonal projection $Pr$, we rewrite the Navier Stokes equation \eqref{4.8} in the operator
  form
  \begin{equation}\label{4.1}
    \frac{du}{dt}+Au+B(u,u)=f(t),\quad  u(0)=u_0\in H,
  \end{equation}
  Here $X=H$, which is a Hilbert space.
 Now suppose  $f$ is a periodic function in $C(\mathbb{R},H)$ and define
  $\theta_tf(\cdot):=f(\cdot +t)$. Then $P=\bigcup  \limits_{t\in
  \mathbb{R}}\theta_tf$ is a compact subset of $C(\mathbb{R},H)$.  Then
  $\varphi(t,u,p):=u(t,u,p)$ is continuous from $\mathbb{R}^+\times H \times C(\mathbb{R},H)\rightarrow
  H$.
  Then the nonautonomous dynamical system $(H, \varphi, (P,\mathbb{R},
  \theta)) $ generated by the Navier-Stokes equation \eqref{4.1} with
  periodic forcing term in $C(\mathbb{R},H)$ has a pullback
  attractor. It is the nontrivial local attractor. By Theorem \ref{3.8}, we conclude that the chain
  recurrent set is not empty. We  can also decomposition the space
  $H$ into two parts: the chain recurrent part $CR_\varphi(p)$ and
  gradient-like part $H-CR_\varphi(p)$.
\end{example}

\vspace{1cm}
 {\bf Acknowledgements.} We thank Martin Rasmussen for helpful comments.
 This research was partly
supported by the NSF grant 0620539, the Cheung Kong Scholars
Program and the K. C. Wong Education Foundation.

\end{document}